\documentclass{amsart}
\usepackage[utf8]{inputenc}

\usepackage{hyperref}

\usepackage{amsmath}
\usepackage{amssymb}
\usepackage{tikz-cd}
\usepackage{enumitem}
\usepackage{bbm}
\usepackage{todonotes}
\usepackage{ifthen}
\usepackage{mathtools}
\usepackage{zref-clever}

\DeclareMathOperator{\dom}{dom}

\DeclareMathOperator{\cf}{cf}

\DeclareMathOperator{\ran}{ran}

\newcommand{\axiom}[1]{\mathsf{#1}}
\newcommand{\CH}{\axiom{CH}}
\newcommand{\PFA}{\axiom{PFA}}
\newcommand{\BPFA}{\axiom{BPFA}}
\newcommand{\MA}{\axiom{MA}}
\newcommand{\BA}{\axiom{BA}}
\newcommand{\ZFC}{\axiom{ZFC}}

\theoremstyle{plain}

\zcRefTypeSetup{thm}{Name-sg=Theorem, name-sg=theorem, Name-pl=Theorems, name-pl=theorems}

\newtheorem{thm}{Theorem}[section]
\newtheorem{lemma}[thm]{Lemma}
\newtheorem{prop}[thm]{Proposition}
\newtheorem{cor}[thm]{Corollary}
\newtheorem{claim}[thm]{Claim}
\newtheorem{fact}[thm]{Fact}

\theoremstyle{definition}
\newtheorem{definition}[thm]{Definition}

\newtheorem{quest}[thm]{Question}

\title{Hagendorf orders}
\author{Jonathan Schilhan and Thilo Weinert}
\address{University of Vienna\\
Institute of Mathematics\\
Kurt Gödel Research Center\\
Kolingasse 14-16\\
1090 Vienna\\
Austria}
\email{jonathan.schilhan@univie.ac.at}
\email{thilo.weinert@univie.ac.at}


\begin{document}

\begin{abstract}
Call a linear order type $\varphi$ \emph{Hagendorf} if it shares two properties with additively indecomposable ordinal numbers without being one itself: $\varphi$ is strictly indecomposable to the right and whenever $\psi < \varphi$, then $\psi$ can be embedded into a proper initial segment of $\varphi$. 

In the 1970s, J. Hagendorf asked whether such types exist. F. Galvin observed that they must be uncountable, and soon thereafter, J. Larson proved that they cannot be scattered. We provide a simplification of her argument that might give additional insight into the $\sigma$-scattered case. We then show that consistently Hagendorf types exist and that this holds in many models of set theory, for instance under $\BA$ or $\MA_{\aleph_1}$. In fact, their non-existence has large cardinal strength. Furthermore, we construct real Hagendorf types from $\diamondsuit$ and from $\PFA$. This marks the first progress on this interesting problem since its description by Larson four dozen years ago.
\end{abstract}

\maketitle

\section{Introduction}

\subsection{Foreword}


The theory of linear orders, as a whole, can be divided into two parts, concerning infinite linear orders and finite linear orders, respectively. However, the latter is more commonly known as ``the theory of numbers", cf. \cite{HardyWright}. It seems natural, then, to argue that linear order types provide a general notion of number: their addition, multiplication, and embeddability relation extend the usual arithmetic operations on the natural numbers and their order.

This paper is concerned with the former part. Its study started more than a century ago with the seminal work of Hausdorff, cf. \cite{Hausdorff}. Over the decades, 
different aspects of linear orders have been studied and some basic problems only have been resolved after being open for decades. Two examples may illustrate this point -- one is Shelah's conjecture from the seventies that the uncountable linear orders consistently admit a basis of cardinality $5$, another is Sierpiński's question from the fifties whether every linear order isomorphic to its cube is isomorphic to its square. The former was proved by Moore in 2006, cf. \cite{Moore2006} and the latter was answered affirmatively by Ervin in 2017, cf. \cite{Ervin}.

In this paper, we are studying a problem first raised by Hagendorf in the seventies of the last century. Roughly, it deals with the question of how closely a linear order type can resemble an ordinal number without actually being well-ordered. Hagendorf's problem is a basic question about linear orders in the best sense: it is easy to state, requires only the most elementary notions of the subject, and yet appears to touch on some of its central difficulties. The problem is somewhat reminiscent of Suslin's hypothesis. A Suslin line, if it exists, is a counterexample to an elementary and compelling characterization of the real line. Similarly, a Hagendorf order is a counterexample to a compelling characterization of well-orders. Despite this simplicity, the problem seems to have received comparatively little attention in the literature. The main question -- whether their existence can be proved from $\ZFC$ alone -- remains open.


\subsection{Basic notions}

Recall that for order types $\varphi$ and $\psi$, the statement $\varphi \leq \psi$ means that for any orders $A$ and $B$ of types $\varphi$ and $\psi$, respectively, $A$ can be embedded into $B$ in an order-preserving way.\footnote{Throughout the paper, we use the term \emph{order-preserving} in the strict sense. That is, $f$ is order-preserving if and only if $x < y$ implies $f(x) < f(y)$.} Types $\varphi$ and $\psi$ are said to be \emph{equimorphic}, often written as $\varphi \equiv \psi$, if $\psi \leq \varphi \leq \psi$.  We write $\varphi \not\leq \psi$ if $\varphi \leq \psi$ is false and $\varphi < \psi$ if $\varphi \leq \psi$ but $\psi \not\leq \varphi$. Furthermore, $\varphi^*$ is the reverse type of $\varphi$, i.e. if $(A, <)$ has order type $\varphi$, then $(A, >)$ has order type $\varphi^*$. The sum $\varphi + \psi$ is defined as the order type of $A \cup B$, where $A$ and $B$ are disjoint, $A$ has type $\varphi$, $B$ has type $\psi$, and every element of $A$ is below every element of $B$. Similarly, if $\alpha$ is an ordinal and $\varphi_\beta$ an order type, for each $\beta < \alpha$, we define the well-ordered sum $\sum_{\beta < \alpha} \varphi_\beta$ and the reverse well-ordered sum $\sum_{\beta < \alpha^*} \varphi_\beta$ in the obvious way. The product $\varphi \psi$ is the type of $B \times A$ under the lexicographic ordering, i.e. the ordering where $(b_0,a_0) \leq (b_1,a_1)$ if and only if $b_0 < b_1$, or $b_0 = b_1$ and $a_0 \leq a_1$.\footnote{Note that the order of operations is reversed for the product. As far as we know this is a historical relic.} This notation is extended to linear orders themselves, so we write $A \leq B$, $A^*$, $A + B$, etc. Ordinal numbers are in general identified with their type under $\in$.


An order type $\varphi$ is \emph{scattered}, if $\eta \not\leq \varphi$, where $\eta$ is the order-type of the rationals, i.e. the unique countable dense linear order without endpoints up to isomorphism. Hausdorff showed that the class of scattered orders is the smallest class of order types containing both $0$ and $1$ which is closed under well-ordered and reverse well-ordered sums, cf. \cite{Hausdorff}. The stage at which an order appears in the recursive procedure of applying these sums is often called its \emph{Hausdorff rank}. The class of scattered order types inherits many properties from that of ordinal numbers and much the same can be said for the even wider class of $\sigma$-scattered order types. These are types of orders $A$ which can be presented as a countable union of subsets $A_n$, $n < \omega$, each of which is scattered in their induced suborder. Famously, Laver proved that the class of $\sigma$-scattered orders is \emph{well-quasi-ordered} under embeddability, settling Fra\"issé's conjecture, cf. \cite{Laver1971}. A partial order $(Q,\leq)$ is well-quasi-ordered if for any sequence $\langle q_n : n < \omega \rangle$ in $Q$, we have that $q_i \leq q_j$, for some $i < j$.

An order type $\tau$ is \emph{indecomposable} if for all $\varphi$ and $\psi$ such that $\tau = \varphi + \psi$ we have $\tau \leq \varphi$ or $\tau \leq \psi$. If we always have that $\varphi < \tau$ and $\tau \leq \psi$, whenever $\psi > 0$, then we call $\tau$ \emph{strictly indecomposable to the right}. If $\varphi^*$ is strictly indecomposable to the right we call $\varphi$ \emph{strictly indecomposable to the left}. If $\tau = \varphi + \psi$ we call $\varphi$ an \emph{initial segment} of $\tau$ and $\psi$ a \emph{final segment} of $\tau$. In this situation $\varphi$ is a \emph{proper initial segment} if $\psi > 0$ and $\psi$ is a \emph{proper final segment} if $\varphi > 0$.

Of course, if $\alpha$ is an ordinal, then every $\psi < \alpha$ embeds in a proper initial segment of $\alpha$, since it is a strictly smaller ordinal. If, moreover, $\alpha$ is indecomposable, then $\alpha$ is strictly indecomposable to the right. Jean Hagendorf asked whether all linear order types sharing these two properties are types of ordinal numbers.


\subsection{Hagendorf's Problem}

\begin{definition}
    We say that an order type $\varphi$ is \emph{Hagendorf} if \begin{enumerate}
        \item $\varphi$ is strictly indecomposable to the right, 
        \item $\psi < \varphi$ implies $\psi + 1 \leq \varphi$,
        \item $\varphi$ is not the type of an ordinal.
    \end{enumerate}

    We say that $(L,<)$ is a \emph{Hagendorf order} if its type is Hagendorf.
\end{definition}

Let us note that the property of being Hagendorf is preserved under equimorphy. Clearly, (2) and (3) are. For (1), this is a standard fact: If $\varphi \leq \tau = \psi + \chi \leq \varphi$, $\chi \neq 0$, then $\tau \not\leq \psi$, as neither $\varphi \leq \psi$, by the strict indecomposability of $\varphi$. The embedding of $\varphi$ into $\psi + \chi$ must meet $\chi$ then, meaning that $\varphi = \psi' + \chi'$, where $\psi' \leq \psi$ and $0 < \chi' \leq \chi$. Thus, $\tau \leq \varphi \leq \chi' \leq \chi$. This will be used implicitly a few times in the paper.

\begin{quest}[{\cite[Problème, p.27]{Hagendorf1982}}]
Do Hagendorf orders exist? Or in its original negated form: 

\begin{quote}
Une chaine strictement insécable à droite dont toute restriction stricte
se plonge dans un intervalle initial propre est-elle un bon-ordre, c’est-à-dire
$X < A \rightarrow X + 1 < A$ entra\^{i}ne-t-il que $\omega^* \not\leq A$?
\end{quote}

\end{quest}

Galvin told Jean Larson about this question and sketched a proof of an affirmative answer within the realm of countable order types (see Theorem~\ref{thm:galvin}). She went on to show that it also holds true within that of scattered order types, cf. \cite{Larson}. In our terminology, Larson showed that no scattered order type is Hagendorf.\footnote{Larson also uses the term ``Hagendorf order'' but without excluding ordinals. We chose to deviate here because the non-ordinal case is of course the one of interest and our results seem more natural to state in this way. A Hagendorf order is a counterexample to Hagendorf's question, just like a Suslin line is a counterexample to Suslin's Hypothesis.} In the abstract of her paper she mentions that this question is particularly interesting within the realms of $\sigma$-scattered and real order types.\footnote{A \emph{real order type} is the type of a suborder of $(\mathbb{R},<)$. }

\subsection{Further background}
Recall that an \emph{Aronszajn line} is an uncountable linear order all whose real subtypes are countable and which does not contain a copy of either $\omega_1$ or $\omega_1^*$. Aronszajn proved that they exist, cf. \cite{Kurepa}. A prominent subclass of Aronszajn lines is that of \emph{Suslin lines} and Suslin's hypothesis states that it is empty. Another example of an Aronszajn line is a \emph{Countryman line} -- a linear order whose square, partially ordered componentwise, is the union of countably many chains. Countryman noted that such orders are Aronszajn lines, and in \cite{Shelah}, Shelah showed that they exist. Galvin noticed that a Countryman line and its reverse only have countable subtypes in common, so a basis for the uncountable linear orders has to contain at least five different uncountable types: $\omega_1$, $\omega_1^*$, a Countryman line and its reverse, and an uncountable real type. That a single order could form a basis for uncountable real types was known quite early on -- Baumgartner's Axiom $(\BA)$ which states that any two $\aleph_1$-dense sets of reals are order-isomorphic was shown consistent by Baumgartner in \cite{Baumgartner}. 

$\BA$ is a consequence of the Bounded Proper Forcing Axiom $\BPFA$, a weakening of the Proper Forcing Axiom $\PFA$, cf. \cite{KruegerMoore}. Thus it holds in a quite canonical class of models in which the continuum hypothesis fails.

\subsection{This Paper}
We start by providing a simpler and shorter proof that scattered orders are not Hagendorf. In contrast to Larson's original argument, almost everything is argued elementarily from the definition of scatteredness as not embedding the rationals. The only notable exception is Lemma~\ref{lem:nogreatest}, which is used in the argument for finite Hausdorff rank (Proposition~\ref{prop:simpdecomp}). Here we apply Laver's profoundly deep well-quasi-orderedness result, but also only to orders of finite Hausdorff rank.

The main motivation for providing such a simplification is the potential of generalizing it to $\sigma$-scattered orders, which is one of Larson's open problems. Form our proof, the main obstacle seems to be to find an appropriate version of Theorem~\ref{thm:sumorreverse} in the $\sigma$-scattered context.

Subsequently we show that Hagendorf types exist in different widely studied models of set theory. We prove in \zcref[S]{thm:BA=>Hagendorf} that their existence follows from Baumgartner's Axiom.\footnote{Let us remark that Baumgartner's result predates Larson's \cite{Larson} by a few years, which seems to be the first mention of Hagendorf's problem in the literature (the former has been submitted in December 1971, the latter in July 1977).} Ervin observed that a similar construction works if one assumes that there is a Countryman line which is minimal for uncountable linear orders, see \zcref[S]{thm:Countryman}. A recent result due to Cummings, Eisworth and Moore lets us conclude that the non-existence of Hagendorf types, in fact, has large cardinal strength, see Corollary~\ref{cor:largecard}. These examples are neither $\sigma$-scattered nor real types, so they do not fall into the case J. Larson deemed particularly interesting. However, relying on a result of Abraham, Rubin and Shelah, essentially the same construction yields a real Hagendorf type, albeit in a not quite so canonical model -- this is \zcref[S]{ARS=>Hagendorf}.

In Section~\ref{sec:diamond}, we show how to construct a real Hagendorf order from the \emph{diamond principle} $\diamondsuit_{\omega_1}$ -- this is \zcref[S]{thm:diamond=>Hagendorf_real}. In particular, they exist in G\"odel's constructible universe and they are consistent with $\CH$. The construction is much more involved than the previous examples and proceeds by a fairly complex transfinite recursion. There are four notable ingredients: the fact that continuum many order-embeddings suffice to cover all order-embeddings on the reals (Fact~\ref{fact:borel}), the Baire Category Theorem, Cantor's back-and-forth construction, and the $\diamondsuit$-sequence. 

Each of these ingredients has an appropriate analogue under $\PFA$, and, in Section~\ref{sec:pfa}, we succeed to modify the construction to suit this context and produce a real Hagendorf order of size $\aleph_2$ -- this is \zcref[S]{thm:PFA=>Hagendorf_real}.

We conclude the paper with the remaining open questions and further comments.

\section{Scattered orders}

Let us begin by reviewing Galvin's short proof that there are no countable Hagendorf orders. The more general argument for scattered orders is somewhat reminiscent of this. 

\begin{thm}[Galvin]\label{thm:galvin}
    No countable linear order is Hagendorf. 
\end{thm}

\begin{proof}
    Let $\varphi$ be the type of a countable Hagendorf order. Then $\varphi \leq \eta$. So if $\eta = \eta + \eta \leq \varphi$, we have that $\varphi$ embeds in a proper initial segment of itself. Thus $\varphi$ must be scattered. By induction, show that every countable ordinal embeds in $\varphi$. If $\alpha \leq \varphi$, then $\alpha < \varphi$, otherwise $\varphi$ is an ordinal. So $\alpha$ embeds in a proper inital segment of $\varphi$, i.e. $\alpha +1 \leq \varphi$. If $\alpha_n \leq \varphi$ for each $n \in \omega$, use the same argument, together with the strict indecomposability of $\varphi$, to see that $\sum_{n< \omega} \alpha_n \leq \varphi$. Any scattered order that embeds every countable ordinal embeds $\omega_1$ or $\omega_1^*$ (see \cite[Theorem 5.28]{Rosenstein1982}, or the more general Theorem~\ref{thm:sumorreverse} below) -- but $\varphi$ is countable.
\end{proof}

This also appears as Exercise 10.4 (14) in Rosenstein's book \cite{Rosenstein1982}.

\subsection{General considerations}
\begin{definition}
    Let $\mathcal{A}$ be a collection of order types and $\alpha$ an ordinal. Then we say that $\tau$ is an \emph{$\alpha$-sum of $\mathcal{A}$} if $\tau = \sum_{i< \alpha} \varphi_i$, where $\varphi_i \in \mathcal{A}$ for each $i < \alpha$. Similarly, we say that $\sum_{i< \alpha^*} \varphi_i$ is an \emph{$\alpha^*$-sum of $\mathcal{A}$}.

    We say that $\tau$ is a \emph{$<\alpha$-sum of $\mathcal{A}$} if it is a $\beta$-sum of $\mathcal{A}$, for some $\beta<\alpha$.
\end{definition}

\begin{lemma}\label{lem:simpdecomp}
    Let $\varphi = \psi_0 + \psi_1$ be an order type, $\alpha$ an indecomposable ordinal and $\mathcal{A}$ a set of order types so that every $<\alpha$-sum of $\mathcal{A}$ embeds into $\varphi$. Then for at least one $i < 2$, every $<\alpha$-sum of $\mathcal{A}$ embeds into $\psi_i$.
\end{lemma}

\begin{proof}
    Let $\tau_0, \tau_1$ be $<\alpha$-sums such that $\tau_0 \not\leq \psi_0$ and $\tau_1 \not\leq \psi_1$. Since $\alpha$ is indecomposable, $\tau = \tau_0 + \tau_1$ is also a $<\alpha$-sum and $\tau \leq \varphi$. An easy contradiction is obtained.
\end{proof}

\begin{prop}\label{prop:simpdecomp}
    Let $\varphi$ be an order type, $\mathcal{A}$ a set of subtypes of $\varphi$ and $\alpha > 1$ an indecomposable ordinal such that every $<\alpha$-sum of $\mathcal{A}$ embeds in $\varphi$. Then one of the following holds. 
 \begin{enumerate}
        \item Every $\alpha$-sum of $\mathcal{A}$ embeds into $\varphi$.
        \item Every $\cf(\alpha)^*$-sum of $\mathcal{A}$ embeds into $\varphi$.
        \item $\varphi = \psi_0 + 1 + \psi_1$, where every $<\alpha$-sum of $\mathcal{A}$ embeds in $\psi_0$ and in $\psi_1$.
    \end{enumerate}
\end{prop}

    \begin{proof}
        Suppose not and let $(L,<)$ be of type $\varphi$. We may assume that $\mathcal{A}$ contains at least some type $\geq 1$, otherwise the claim is trivial. Let $I_0$ be the largest initial segment (possibly empty) of $L$ such that for each $l \in I_0$, $(-\infty, l)$ does not embed each $<\alpha$-sum of $\mathcal{A}$. Similarly, let $I_1$ be the largest final segment (again, possibly empty) such that for each $l \in I_1$, $(l, \infty)$ does not embed each such type. Then note that there is no $l \in L$ with $I_0 < l < I_1$, since otherwise (3) is immediately seen to hold true with $\psi_0$ the type of $I_0$ and $\psi_1$ the type of $I_1$.\footnote{$I_0 < l < I_1$ means $\forall a \in I_0 \forall b \in I_1 (a < l < b)$, as expected.}

        By Lemma~\ref{lem:simpdecomp}, one of $I_0$, $I_1$ must embed each $<\alpha$-sum. Suppose for now this is the case for $I_0$. $I_0$ can't be of the form $(-\infty, l]$, otherwise $(-\infty, l)$ already embeds all $<\alpha$-sums. Let $\theta = \cf(I_0)$ and $\langle l_\xi : \xi < \theta \rangle$ be cofinal in $I_0$.

        \begin{itemize}
            \item Case 1: $\theta \geq \cf(\alpha)$. By Lemma~\ref{lem:simpdecomp} again, each final segment $[l_\xi, \infty) \cap I_0$ of $I_0$ must embed each $<\alpha$-sum; in particular, also in a non-cofinal way. We easily obtain that each $\alpha$-sum embeds in $I_0$, thus in $L$. 
            \item Case 2: $\theta < \cf(\alpha)$. For each $\xi < \theta$, let $\tau_\xi$ be a $<\alpha$-sum not embedding into $(-\infty, l_\xi)$. Then $\tau = \sum_{\xi < \theta} \tau_\xi + 1$ is still a subtype of a $<\alpha$-sum and embeds into $I_0$. We obtain an immediate contradiction to the cofinality of $\langle l_\xi : \xi < \theta \rangle$.
        \end{itemize}

        Both cases yield a contradiction. Now suppose $I_1$ embeds every $<\alpha$-sum. The argument is analogous. Let $\langle l_\xi : \xi < \theta \rangle$ be coinitial in $I_1$, noting that $I_1$ can't have a minimal element. If $\theta \geq \cf(\alpha)$, we find that every $\cf(\alpha)^*$-sum embeds in $I_1$, because $1 + \psi$ embeds for every $\psi \in \mathcal{A}$. If $\theta < \cf(\alpha)$, we find that a subtype $1 + \sum_{\xi < \theta} \tau_\xi$ of a $<\alpha$-sum can't embed into $I_1$.
    \end{proof}

\begin{thm}\label{thm:sumorreverse}
    Let $\varphi$ be an order type, $\mathcal{A}$ a set of subtypes of $\varphi$ and $\alpha > 1$ an indecomposable ordinal such that every $<\alpha$-sum of $\mathcal{A}$ embeds in $\varphi$. Then one of the following holds. 
    \begin{enumerate}
        \item Every $\alpha$-sum of $\mathcal{A}$ embeds into $\varphi$.
        \item Every $\cf(\alpha)^*$-sum of $\mathcal{A}$ embeds into $\varphi$.
        \item $\eta \leq \varphi$, i.e. $\varphi$ is not scattered.
    \end{enumerate}
\end{thm}

\begin{proof}
    Suppose both (1) and (2) are not the case. Note that if $\psi_0$, $\psi_1$ are as in (3) of Proposition~\ref{prop:simpdecomp}, (1) and (2) must still fail for $\varphi$ replaced by $\psi_0$ and $\psi_1$. Thus the proposition can be applied inductively. Given $(L,<)$ of type $\varphi$, recursively define, for each $s \in 2^{<\omega}$, a convex subset $I(s)$ of $L$ embedding every $<\alpha$-sum of $\mathcal{A}$ and $l_s \in I(s)$, as follows. Let $I_\emptyset = L$. Having defined $I_s$, pick $l_s \in I_s$ so that both $I_{s^\frown 0} := (-\infty, l_s) \cap I_s$ and $I_{s^\frown 1}:= (l_s, \infty) \cap I_s$ embed all $<\alpha$-sums. Clearly, $\{ l_s : s \in 2^{<\omega} \}$ is a countable dense-in-itself suborder of $L$. \end{proof}

\begin{lemma}\label{lem:nogreatest}
    Let $(L,<)$ be a $\sigma$-scattered order without a greatest element. Then there there is $l \in L$, such that for each $k > l$, $\omega \times [l,k) \leq L$.\footnote{Recall that $\omega \times [l,k)$ carries the lexicographic ordering. So if $[l,k)$ has type $\varphi$, then $\omega \times [l,k)$ has type $\varphi \omega$.}
\end{lemma}

\begin{proof}
    Let $\theta$ be the cofinality of $(L, <)$ and $\langle l_\alpha : \alpha < \theta \rangle$ a cofinal and increasing sequence in $L$. Define a colouring $c \colon [\theta]^2 \to 2$, where $c(\alpha < \beta) = 0$, if $[l_\alpha, l_\beta)$ embeds in $[l_\beta, \infty)$, and $c(\alpha < \beta) = 1$ otherwise. By the Erd\H{o}s-Dushnik-Miller Theorem (see \cite[Theorem 5.22]{DushnikMiller1941}), there is a homogeneous subset $\{\alpha_i : i < \delta \}$ of $\theta$, written in increasing order, with either colour $1$ and $\delta = \omega$, or colour $0$ and $\delta = \theta$. The former case cannot occur. Otherwise, letting $\varphi_i$ be the type of $[l_{\alpha_i}, l_{\alpha_{i+1}})$, $\varphi_i \not\leq \varphi_j$, for all $i < j$. This contradicts the well-quasi-orderedness of the $\sigma$-scattered orders. So the latter case occurs. It is easy to see that $l = l_{\alpha_0}$ works, noting that for any $i < \theta$, there is $j > i$ so that $[l_{\alpha_0}, l_{\alpha_i})$ embeds in $[l_{\alpha_i}, l_{\alpha_j})$.
\end{proof}

\subsection{No scattered Hagendorf order} To finish our proof that scattered orders aren't Hagendorf, we first show that this is the case for orders of finite Hausdorff rank. This is the only place where the previous lemma is applied.

\begin{definition}\label{def:simpleorders}
    We denote with $\mathcal{S}_0$ the class of ordinals and reverse ordinals. Whenever $\mathcal{S}_n$ is defined, $\mathcal{S}_{n+1}$ consists of all the well-ordered and reverse well-ordered sums of orders in $\mathcal{S}_n$.
\end{definition}

\begin{prop}\label{lem:simplenoHO}
    Let $\varphi \in \mathcal{S}_n$, for some $n \in \omega$. Then $\varphi$ is not a Hagendorf order. 
\end{prop}

\begin{proof}
    We show by induction on $n$, that for any non-zero $\varphi \in \mathcal{S}_{n+1}$, there is $\varphi' \leq \varphi$, $\varphi' \in \mathcal{S}_n$, with \begin{equation}\label{eq:prop}
        \varphi' \cdot \omega \not\leq \varphi.
    \end{equation} Then, if $\varphi$ is Hagendorf, $\varphi'$ must be as well. Otherwise, if $\varphi'$ is not Hagendorf, then $\varphi'$ is not equimorphic to $\varphi$, so $\varphi'$ is a strict subtype of $\varphi$. But then we can show that $\varphi' \cdot \omega \leq \varphi$, using the Hagendorfness of $\varphi$ -- contradiction. This shows that if $n$ is minimal with a Hagendorf order in $\mathcal{S}_n$, then $n$ must be $0$, which is impossible.

    If $\varphi$ is an ordinal, stipulate that $\varphi' = \varphi$. If $\varphi$ is a reverse ordinal, let $\varphi' = 1$. Then (\ref{eq:prop}) is clearly satisfied. Now suppose that $\varphi \in \mathcal{S}_{n+1}$, say $\varphi = \sum_{\xi < \alpha} \varphi_\xi$. We define $\varphi' = \sum_{\xi < \alpha} \varphi_\xi'$. First note that indeed $\varphi' \in \mathcal{S}_n$. When $n = 0$, this is the case because $\varphi'$ is an ordinal. Otherwise, use the inductive assumption that $\varphi_\xi' \in \mathcal{S}_{n-1}$ for each $\xi$. Next, we show by induction on $\alpha$ that (\ref{eq:prop}) is satisfied. The case $\alpha = 1$ is clear, and the successor case follows from the simple observation that, if $\psi' \cdot \omega \not\leq \psi$ and $\chi' \cdot \omega \not\leq \chi$, then also $(\psi' + \chi') \cdot \omega \not\leq \psi + \chi$. For the limit case, apply Lemma~\ref{lem:nogreatest} and find $\beta < \alpha$, so that for each $\gamma \in [\beta, \alpha)$, $(\sum_{\beta \leq \xi < \gamma} \varphi_\xi')\cdot \omega \leq \sum_{\beta \leq \xi < \alpha} \varphi_\xi'$. If $(\sum_{\beta \leq \xi < \alpha} \varphi_\xi')\cdot \omega \leq \sum_{\beta \leq \xi < \alpha} \varphi_\xi$, then $$(\sum_{\beta \leq \xi < \gamma} \varphi_\xi')\cdot \omega \leq \sum_{\beta \leq \xi < \alpha} \varphi_\xi' \leq \sum_{\beta \leq \xi < \gamma} \varphi_\xi,$$ for some $\gamma < \alpha$, contradicting the inductive assumption. By writing $\varphi$ as $\sum_{\xi < \beta} \varphi_\xi + \sum_{\beta \leq \xi < \alpha} \varphi_\xi$, we find that $\varphi'$ works with the same argument as in the successor case. 

    The case of reverse sums $\varphi = \sum_{\xi < \alpha^*} \varphi_\xi$ is even simpler. We only need to consider $\alpha \geq \omega$, since otherwise, this is a well-ordered sum. If $n = 0$, let $\varphi' = \alpha^*$. For $n \geq 1$, we let $\varphi' =  \sum_{\xi < \alpha^*} \varphi_\xi'$. We leave the details to the reader.
    \end{proof}

\begin{thm}
    There is no scattered Hagendorf order.
\end{thm}

\begin{proof}
Suppose $\varphi$ is the order-type of a scattered Hagendorf order. For each $n \in \omega$, let $\mathcal{A}_n := \{\psi \in \mathcal{S}_n : \psi \leq \varphi \}$, and let $\alpha_{0,n}$ be the least ordinal such that some $\alpha_{0,n}$-sum of $\mathcal{A}_n$ does not embed in $\varphi$. Similarly, let $\alpha_{1,n}$ be least such that some $(\alpha_{1,n})^*$-sum of $\mathcal{A}_n$ does not embed in $\varphi$. Note that $\langle \alpha_{i,n} : n \in \omega \rangle$ is non-increasing, for $i=0,1$. Thus, there is a point $n$ from which both of these sequences are constant. We define $\theta_0 := \alpha_{0,n}$ and $\theta_1 := \alpha_{1,n}$.

\begin{claim}
    $\theta_0$ and $\theta_1$ are regular cardinals.
\end{claim}

\begin{proof}
     First, $\theta_0$ is a limit ordinal. Namely, whenever $\psi = \sum_{\xi < \alpha} \psi_\xi$ is sum of $\mathcal{A}_n$ that does embed in $\varphi$, we have that $\psi < \varphi$. Otherwise, $\varphi \leq \psi$ and $\psi$ is also Hagendorf, contradicting Proposition~\ref{lem:simplenoHO}. Thus $\psi$ embeds into a proper initial segment. Any $\nu \in \mathcal{A}_n$ embeds into any final segment of $\varphi$, by strict indecomposability. But then also $\psi + \nu = \sum_{\xi < \alpha} \psi_\xi + \nu$ embeds in $\varphi$. So $\theta_0 = \alpha +1$ is impossible. Similarly, $\theta_1$ is a limit.
     
     Now suppose that $\cf(\theta_0) < \theta_0$ and that $\sum_{\xi < \theta_0} \psi_\xi$ is a $\theta_0$-sum of $\mathcal{A}_n$ that does not embed in $\varphi$. By assumption, $\sum_{\xi < \delta} \psi_\xi \in \mathcal{A}_{n+1}$ for each $\delta < \theta_0$. Clearly then, we find a $\cf(\theta_0)$-sum of $\mathcal{A}_{n+1}$ not embedding in $\varphi$. But then $\alpha_{0,n+1} \leq \cf(\theta) < \alpha_{0,n} = \theta_0$. The argument for $\theta_1$ is completely analogous.
\end{proof}

Theorem~\ref{thm:sumorreverse} implies that $\theta_0 < \theta_1$ and that $\theta_1 < \theta_0$ -- contradiction.  \end{proof}

\section{First consistent examples}

\subsection{Baumgartner's Axiom}

Recall that a subset $A \subseteq \mathbb{R}$ is said to be \emph{$\aleph_1$-dense} if the intersection of $A$ with any non-empty open subset of $\mathbb{R}$ is of size $\aleph_1$.

\begin{definition}[Baumgartner's Axiom]
    $\BA$ is the statement that any two $\aleph_1$-dense suborders of $(\mathbb{R}, <)$ are isomorphic.
\end{definition}

\begin{thm}\label{thm:BA=>Hagendorf}
    $\BA$ implies that there is a Hagendorf order of size $\aleph_1$.
\end{thm}

The theorem is a consequence of the following propositions.

\begin{prop}\label{prop:BAhag}
    Let $\varphi \geq \eta$ be an order type such that for every $n \in \omega$,
    \begin{enumerate}  
        \item \label{eq:first} $\varphi^n < \varphi^{n+1}$,
        \item \label{eq:second} whenever $\psi < \varphi^{n+1}$, then $\psi \leq \eta \varphi^n$.
    \end{enumerate}
    Then $\sum_{n < \omega} \varphi^n$ is Hagendorf.
\end{prop}

\begin{proof}
    Condition (1) clearly implies that $\sum_{n < \omega} \varphi^n$ is strictly indecomposable to the right. Also, this can't be an ordinal since it embeds the rationals. Now let $X$ be a linear order of type $\varphi$ and suppose that $A$ is a suborder of $\sum_{n < \omega} X^n$. Call $A_n$ the intersection of $A$ with the copy of $X^n$.
    
    \underline{Case 1}: For any $n \geq 1$, there is $m > n$ such that $X^n \leq A_m$. Then clearly $\sum_{n < \omega} X^n \leq A$.

    \underline{Case 2}: There is $n$ such that $X^n \not\leq A_m$, for any $m > n$. Of course $A_m \leq X^m \leq X^n$, for $m \leq n$. Now let $m > n$ and $k$ be least so that $A_m \leq X^{k} \times \mathbb{Q}$. We claim that $X^k \leq A_m$, so $k < n$ and $A_m \leq X^n \times \mathbb{Q}$. To see this, let $e \colon A_m \to X^{k} \times \mathbb{Q}$ be an order-embedding and let $A_m' \subseteq A_m$ be such that for each $a \in A_m$, there is exactly one $a' \in A_m'$, with $e(a) \restriction k = e(a') \restriction k$.\footnote{$e(a) \in X^{k} \times \mathbb{Q}$ is a tuple of length $k+1$, so $e(a) \restriction k$ is it's restriction to the factor $X^k$.} We have that $A_m' \leq X^{k}$ by construction. By (2), either $X^k \leq A_m'$, so also $X^k \leq A_m$, or $A_m' \leq X^{k-1} \times \mathbb{Q}$.\footnote{Of course if $k = 0$, then $X^k \leq A_m$.} In the latter case, it is easy to see that $A_m \leq X^{k-1} \times \mathbb{Q} \times \mathbb{Q} \cong X^{k-1} \times \mathbb{Q}$. 
    This poses a contradiction to the minimality of $k$.
    
    Finally, we have shown that $A_m \leq X^n \times \mathbb{Q}$ for every $m$. Since $\mathbb{Q} \leq X$, $A \cong \sum_{m < \omega} A_m \leq \omega \times (X^{n} \times \mathbb{Q}) \leq X^{n+2}$ which is contained in a proper initial segment of $\sum_{n < \omega} X^n$.
    \end{proof}

\begin{prop}
\label{prop:prod-ineq}
    Let $A_0, \dots, A_{n+1}$ be uncountable suborders of $(\mathbb{R},<)$. Then $A_0 \times \dots \times A_n < A_0 \times \dots \times A_n \times A_{n+1}$.
\end{prop}

\begin{proof}
    This is an induction on $n \in \omega$. Clearly $A_0 < A_0 \times A_1$ as $A_0 \times A_1$ contains uncountably many disjoint non-empty open intervals. This is impossible for a suborder of the reals. Now suppose towards a contradiction that $e \colon A_0 \times \dots \times A_{n+1} \to A_0 \times \dots \times A_n$ is an order-embedding. For any $a \in A_0$, define $$E(a) := \{e(a^\frown \bar x )_0 : \bar x \in A_1 \times \dots\times A_{n+1}\} \subseteq A_0.\footnote{$e(a^\frown \bar x )_i$ stands for the $i$'th component of $e(a^\frown \bar x)$.}$$

    \underline{Case 1}: There is $a \in A_0$ such that $E(a)$ is a singleton. Then note that $$f(\bar x) := e(a^\frown \bar x ) \restriction \{1, \dots, n \}$$ defines an order embedding of $A_{1} \times \dots \times A_{n+1}$ into $A_1 \times \dots \times A_{n}$, which is impossible by the inductive assumption.
    
    \underline{Case 2}: For any $a \in A_0$, $E(a)$ contains at least two elements. Pick a rational $r_a$ that lies between two elements of $E(a)$. Then note that $a \mapsto r_a$ defines an injection of $A_0$ into the rationals. But $A_0$ is assumed to be uncountable -- contradiction. 
\end{proof}

\begin{prop}
    Assume $\BA$ and let $\varphi$ be the order type of an $\aleph_1$-dense suborder of $\mathbb{R}$. Whenever $\psi < \varphi^{n+1}$, then $\psi \leq \eta \varphi^n$.
\end{prop}

\begin{proof}
    By induction on $n$. This is true for $n = 0$, since any set of reals is either countable, so embeds in the rationals, or contains a copy of an $\aleph_1$-dense set (see e.g. \cite[Corollary 5.40]{ErvinMarconeWeinert}). So assume $n \geq 1$ now. 
    
    Let $A \subseteq B^{n+1}$, where $B$ is an $\aleph_1$-dense subset of $\mathbb{R}$. Then let $\pi[A] \subseteq B$ be the projection of $A$ onto the first coordinate. For any $a \in \pi[A]$, let $A_a := \{ \bar x \in B^n : a^\frown \bar x \in A \}$. Define $C := \{a \in \pi[A] : B^n \leq A_a  \}.$

    \underline{Case 1}: $C$ is uncountable. Then $C$ contains a copy of an $\aleph_1$-dense set of reals, thus a copy of $B$. It is easy to see that $B^{n+1} \leq A$.
    
    \underline{Case 2}: $C$ is countable. For each $a \in \pi[A]$, let $L_a := B$ if $a \in C$ and $L_a := \{ 0\}$ otherwise. We note that $\sum_{a \in \pi[A]} L_a \leq B$, since this is a separable linear order, thus embeds in $\mathbb{R}$, and has size at most $\aleph_1$.\footnote{An order $L$ is \emph{separable} if there is a countable $Q \subseteq L$, such that for any $x < y$, $[x,y] \cap Q \neq \emptyset$. $Q$ embeds in the rationals and each Dedekind cut in $Q$ is filled by at most one element of $L$. Thus, $L$ embeds in the reals.} 
    
    For $a \in \pi[A] \setminus C$, since $A_a < B^n$, we have that $A_a \leq B^{n-1} \times \mathbb{Q} \cong L_a \times B^{n-1} \times \mathbb{Q}$ by induction. For $a \in C$, $A_a \leq B^n = L_a \times B^{n-1} \leq L_a \times B^{n-1} \times \mathbb{Q}$. Thus \begin{align*}
        A  \cong \sum_{a \in \pi[A]} A_a &\leq \sum_{a \in \pi[A]} (L_a \times B^{n-1} \times \mathbb{Q})\\  &\cong(\sum_{a \in \pi[A]} L_a) \times B^{n-1} \times \mathbb{Q}\\ &\leq B \times B^{n-1} \times \mathbb{Q} \cong B^n \times \mathbb{Q}.
    \end{align*}\end{proof}

This shows that under Baumgartner's Axiom, the order type $\varphi$ of an $\aleph_1$-dense set of reals satisfies the conditions of Proposition~\ref{prop:BAhag}. \eqref{eq:second} is fulfilled because of the calculation above and Proposition \ref{prop:prod-ineq} shows that \eqref{eq:first} is.

This proves Theorem~\ref{thm:BA=>Hagendorf}. We also note that it is consistent that there is no Hagendorf suborder of $(\mathbb{R}, <)$ of size $\aleph_1$.


\begin{thm}\label{thm:BA=>noaleph1}
    Assuming $\BA$, there is no strictly indecomposable suborder of $\mathbb{R}$ of size $\aleph_1$.
\end{thm}

\begin{proof}
    Under $\BA$, any uncountable set of reals contains a copy of any set of reals of size $\aleph_1$. Any uncountable set of reals contains an uncountable proper initial segment and an uncountable proper final segment (see for example the argument above).
\end{proof}

\subsection{Minimal Countryman lines}
After discussing the previous result with Garett Ervin, they pointed out the following other consistent example, which we include with their permission. 

For the precise definition of a \emph{Countryman line} we refer to \cite{Moore2006}. The only properties of a Countryman line $C$ that we use are that $C$ is uncountable, $C \not\leq C^*$, $C^* \not\leq C$, $\omega_1 \not\leq C$, and $\omega_1^* \not\leq C$.


\begin{thm}\label{thm:Countryman}
    Assume that there is a Countryman line $(C, <)$ which is minimal for uncountable linear orders, i.e. any uncountable suborder of $C$ embeds $C$. Then there is a Hagendorf order of size $\aleph_1$.
\end{thm}

The assumption of the theorem holds under $\MA_{\aleph_1}$ (see \cite[Theorem 2.1.12]{Todorcevic}) and under $\diamondsuit$ (see \cite{Cummings}).

\begin{proof}
    We note that $X := \sum_{n \in \omega} (C + C^*) = C + C^* + C + C^* + \dots$ is Hagendorf. 
    
    Clearly $X$ embeds into every one of its (non-empty) final segments. If $X$ were to embed into a proper initial segment, we would obtain that $C \leq C^*$ or $C^* \leq C$ which is impossible. Thus $X$ is strictly indecomposable to the right. 
    
    Next, let $A \subseteq X$ be arbitrary and consider $A$'s intersection with the copies of $C$ and $C^*$. If $A$ is uncountable on infinitely many copies of $C$ and on infinitely many copies of $C^*$, then clearly $X$ embeds into $A$ again, by the minimality of $C$. On the other hand, suppose that $A$ is countable on all but finitely many copies of, say, $C^*$ (the case of $C$ is analogous). $C$ embeds every countable linear order, as $C$ is non-scattered (otherwise, either $\omega_1\leq C$, or $\omega_1^* \leq C$, cf. \cite[Theorem 5.28]{Rosenstein1982}). Thus, $A$ embeds into $\sum_{k \leq n} (C+C^*) + \omega \times C$, for some $n \in \omega$. To conclude that $A$ embeds into a proper initial segment of $X$ it suffices to observe the following:

    \begin{claim}
        $\omega \times C$ embeds into $C$.
    \end{claim}

    \begin{proof}
        Consider $I = \{ x \in C : \vert (-\infty, x] \vert  \leq \aleph_0 \}$ and $J = \{ x \in C :  \vert [x, \infty)\vert \leq \aleph_0 \}$. Note that $I$ and $J$ are disjoint initial and final segments of $C$, respectively. $I$ is countable, otherwise it embeds $\omega_1$. Similarly, $J$ is countable. Thus there is $x \in C \setminus (I \cup J)$ and both $(-\infty, x]$ and $[x, \infty)$ are uncountable. It follows that $C$ can be written as $C_0 + C'$, where both $C_0$ and $C'$ are uncountable. In the same vein, we can then write $C'$ as $C_1 + C''$, with $C_1$ and $C''$ uncountable. Continuing like this, we find that $\sum_{n \in\omega} C_n \leq C$, where each $C_n$ is an uncountable suborder of $C$. By the minimality assumption on $C$, $\omega \times C \leq \sum_{n \in\omega} C_n \leq C$.
    \end{proof}\end{proof}

In \cite{Cummings}, the authors show that a combinatorial principle called ``\emph{diamond in the square}" at an infinite cardinal $\kappa$ implies the existence of a \emph{$\kappa^+$-Countryman line $C$} that is minimal with respect to non-$\sigma$-scattered orders. The properties relevant to us are that $C$ and $C^*$ are incomparable and that for every suborder $X \subseteq C$, either $C \leq X$, which happens exactly when $\vert X \vert = \vert C \vert = \kappa^+$, or $X \leq C^*$, when $\vert X \vert \leq \kappa$ and in which case $X$ is $\sigma$-scattered (see \cite[Proposition 4.2]{Cummings} and note that a $\kappa^+$-Aronzsajn line embeds every $\sigma$-scattered order of size $\leq \kappa$). Essentially the same argument as that of Theorem~\ref{thm:Countryman} shows that $C + C^* + C + \dots$ is Hagendorf. Now it is noted that the failure of diamond in the square at every $\kappa$ implies an inner model with a measurable $\mu$ of Mitchell order $\mu^{++}$ (see \cite[Corollary 1.15]{Cummings}). Thus, the non-existence of Hagendorf orders has large cardinal strength. 

\begin{cor}\label{cor:largecard}
    Assume that there is no Hagendorf order. Then there is an inner model with a measurable cardinal $\mu$ of Mitchell order $\mu^{++}$.
\end{cor}

By \cite{Moore2007}, it is consistent that $\omega_1$ and $\omega_1^*$ are the only minimal uncountable linear orders. Later, it was shown that relative to a supercompact cardinal, it is consistent that there are no minimal non-$\sigma$-scattered orders (see \cite{LameiRamandi2018}). Thus the arguments above all seem to be insufficient to provide $\ZFC$ examples for Hagendorf orders. Minimality is of course somewhat inherent in Hagendorf orders, as they are minimal above an increasing sequence of subtypes (their proper initial segments).

\subsection{A consistent real Hagendorf type}
None of the examples we have given so far are real order types. We can use a result of Abraham, Rubin and Shelah as a black-box to show that consistently real Hagendorf orders of size $\aleph_1$ do exist.

\begin{thm}\label{ARS=>Hagendorf}
    It is consistent with $2^{\aleph_0} = \aleph_2$ that there is a real Hagendorf order of size $\aleph_1$.
\end{thm}

\begin{proof}
   According to \cite{Abraham}, there is a model of $2^{\aleph_0} = \aleph_2$ with $A, B \in [\mathbb{R}]^{\aleph_1}$ such that $A$ and $B$ are incomparable, i.e. $A \not\leq B$ and $B \not\leq A$, and minimal for uncountable linear orders. Run the same argument as in Theorem~\ref{thm:Countryman}.
\end{proof}

\section{A real Hagendorf order from \texorpdfstring{$\diamondsuit$}{diamond}}\label{sec:diamond}

The following fact is one of the main tools when dealing with real order types. It was frequently used by Sierpiński, for instance in the construction of a rigid real order type, cf. \cite{Sierpinski1932}, \cite{Kuratowski1926}. Dushnik and Miller implicitly used it to construct an \emph{exact} suborder of $\mathbb{R}$, cf. \cite{DushnikMiller1940}.\footnote{An order is exact if the only self embedding is the identity. This term appears in \cite{Ginsburg1955}.}

\begin{fact}\label{fact:borel}
    Let $X \subseteq \mathbb{R}$ and $f \colon X \to \mathbb{R}$ be order-preserving. Then there is an order-preserving Borel map $\hat f \colon \hat X \to \mathbb{R}$ extending $f$. In particular, there are continuum many order-preserving maps on the reals covering any of the $2^{\mathfrak{c}}$-many such maps.\footnote{$\mathfrak{c}$ is the cardinality of the continuum.}
\end{fact}

\begin{proof}
Consider the extended real line $[-\infty, \infty]$ that is homeomorphic with $[0,1]$. Define $\tilde f \colon \mathbb{R} \to [-\infty, \infty]$, where $\tilde f(x) = \sup \{ f(a) : a \leq x \}$.\footnote{The supremum of the empty set is $-\infty$.} Clearly $f \subseteq \tilde f$ and we have that $x \leq y$ implies $\tilde f(x) \leq \tilde f(y)$, i.e. $\tilde f$ is non-decreasing. It is a Borel map, since the preimage of an interval is an interval. Let $\mathcal{I}$ be the collection of maximal non-singleton intervals on which $\tilde f$ is constant. Then $\mathcal{I}$ is at most countable, since a distinct rational can be found in each such interval, and $A = \bigcup \mathcal{I}$ is Borel. Since $I \cap X$ can have at most one element for each $I \in \mathcal{I}$, $A \cap X$ is countable. Let $\hat X = (\mathbb{R} \setminus A) \cup (A \cap X)$ and $\hat f = \tilde f \restriction \hat X$. It is straightforward to check that $\hat f$ maps to $\mathbb{R}$ and is injective, thus order-preserving.  
\end{proof}

We first investigated what can be said under $\CH$. Originally, we thought a negative result would be in reach, due to the following: 

\begin{thm}
    Assume $\CH$ and let $X \subseteq \mathbb{R}$ be Hagendorf. Then there is a countable set $C \subseteq X$ such that $X \setminus C < X$.
\end{thm}

Or more generally: 

\begin{thm}\label{thm:smallembed}
    Let $X \subseteq \mathbb{R}$ be Hagendorf. Then there is a set $C \subseteq X$ of size $< \mathfrak{c}$ such that $X \setminus C < X$.
\end{thm}

Of course this is trivial when $\vert X \vert < \mathfrak{c}$. But under $\CH$, where a real Hagendorf order must have size $\mathfrak{c} = \aleph_1$, the situation seems a bit absurd. Being Hagendorf imposes a dichotomy on suborders of $X$. Either a suborder is ``small'', meaning it appears in a proper initial segment, or it is as ``large'' as $X$. One would naively expect that removing only countably many element should keep the order ``large''.




\begin{proof}
    Suppose $X \subseteq \mathbb{R}$ is Hagendorf and for every countable $C \subseteq X$, $X \setminus C \not< X$, i.e. $X \leq X \setminus C$. Without loss of generality we can also assume that $X$ is cofinal in $\mathbb{R}$. This will simplify some notation. Let $\langle (f_\alpha, n_\alpha) : \alpha < \omega_1 \rangle$ enumerate all pairs $(f,n)$ of functions given by the fact above and natural numbers $n$. We will recursively construct a sequence $\langle x_\alpha : \alpha < \omega_1 \rangle$ and $\langle C_\alpha : \alpha < \omega_1 \rangle$ as follows. Each $x_\alpha$ will be an element of $X$ and $C_\alpha$ will be a countable subset of $X$. Suppose $x_\xi$, $C_\xi$ have been constructed for all $\xi < \alpha$. Let $C = \bigcup_{\xi < \alpha} C_\xi$. Given $f_\alpha$ and $n_\alpha$, we have a few options:
    \begin{enumerate}
        \item If $X \setminus C \not \subseteq \dom f_\alpha$, let $C_\alpha = C$ and $x_\alpha \in (X \setminus C) \setminus \dom f_\alpha$ be arbitrary.
        \item If $X \setminus C \subseteq \dom f_\alpha$, but $f[X\setminus C] \not\subseteq X$, let $C_\alpha = C$ and $x_\alpha \in X\setminus C$ be such that $f_\alpha(x_\alpha) \notin X$.
        \item If $X \setminus C \subseteq \dom f_\alpha$, and $f_\alpha[X\setminus C] \subseteq X$, note that $f[X\setminus C]$ must be cofinal in $\mathbb{R}$. Otherwise $f_\alpha$ witnesses that $X \setminus C$ embeds into an initial segment of $X$, and thus $X \setminus C < X$. Since every final segment of $X$, also of $X \setminus C$, is uncountable, we can find $x_\alpha \in X \setminus C$ such that $f_\alpha(x_\alpha) > n_\alpha$. Also, we can find $y \in X \setminus C$ such that and $f_\alpha(y) \notin \{ x_\xi : \xi < \alpha \}$. Let $C_\alpha = C \cup \{ f_\alpha(y)\}$. 
    \end{enumerate}

    Finally, let $A = \{x_\alpha : \alpha < \omega_1 \}$. Note that $A \cap \bigcup_{\alpha \in \omega_1} C_\alpha = \emptyset$.

    \begin{claim}
        $A$ does not embed in any initial segment of $X$.
    \end{claim}
    
    \begin{proof}
        Otherwise, there is some $f$ and $n$ so that $f \colon A \to X \cap (-\infty, n)$ is order-preserving. Then there is $\alpha$ so that $f \subseteq f_\alpha$ and $n = n_\alpha$. At step $\alpha$ of the construction, neither option (1) nor option (2) could have happened as $x_\alpha \in \dom f \subseteq \dom f_\alpha$ and $f_\alpha(x_\alpha) \in X$. Thus (3) must have been the case. But then $f(x_\alpha) = f_\alpha(x_\alpha) > n_\alpha = n$, which also contradicts our assumption on $f$. 
    \end{proof}
    \begin{claim}
        $X$ does not embed into $A$.
    \end{claim}
\begin{proof}
    The proof is essentially the same as above starting from $f \colon X \to A$ and any $n$, noting that for the same reasons (1) and (2) are impossible, while in (3) we have guaranteed that $f(y) \notin A$ for some $y \in X$.
\end{proof}

Thus $A$ is a counterexample to $X$ being Hagendorf, posing a contradiction.
\end{proof}

\begin{thm}\label{thm:diamond=>Hagendorf_real}
    Assuming $\diamondsuit_{\omega_1}$, there is a Hagendorf order $X \subseteq \mathbb{R}$.
\end{thm}

\begin{proof}\label{thm:diamond}
      We will construct the Hagendorf order as a subset of $R = (0, \infty) \setminus \mathbb{N}$. Fix any countably dense subset $Y$ of $R$. There will be a countable set $\mathcal{G}$ of functions $g \colon R \to R$, and we recursively construct a sequence $\langle X_\alpha, Z_\alpha, \mathcal{H}_\alpha : \alpha < \omega_1 \rangle$ that is $\subseteq$-increasing in all variables, where
    \begin{itemize}
        \item $X_\alpha \subseteq R$ is countable dense,
        \item $Z_\alpha \subseteq R$ is countable,
        \item $X_\alpha \cap (Y \cup Z_\alpha) = \emptyset$,
        \item $\mathcal{G}$ consists of continuous progressive (for all $x$, $g(x) > x$) order-preserving functions $g \colon R \to R$, with $g[Y] \subseteq Y$, i.e. $Y$ is closed under $\mathcal{G}$,
        \item $\mathcal{H}_\alpha$ consists of triples $(h, C, r)$, where $r$ is a natural number, $C$ is a countable dense subset of $X_\alpha \cap (r, \infty)$ and $h \colon R \to R \cap (0, r+1)$ is a continuous order-preserving map so that $ h(x) = x$ for all $x < r$.
    \end{itemize}

    Whenever $\mathcal{A}$ is a set of functions $R \to R$ and $w \in \mathcal{A}^{<\omega}$ is a word with letters in $\mathcal{A}$, we write $\langle w \rangle$ for the function $w(\vert w \vert -1) \circ \dots \circ w(0) \colon R \to R$. The following are satisfied at each step: 

    \begin{enumerate}
        \item For any $g \in \mathcal{G}$, $g[X_\alpha] \subseteq X_\alpha$, i.e. also $X_\alpha$ is closed under $\mathcal{G}$.
        \item For any $(h,C,r) \in \mathcal{H}_\alpha$, $h[X_\alpha \setminus C] \subseteq X_\alpha$, $h[Y] \subseteq Y$ and $h[C] \subseteq Y$.
        \item For any \begin{enumerate}
    \item $(h, C, r) \in \mathcal{H}_\alpha$,
    \item word $w \in (\mathcal{G} \cup \dom \mathcal{H}_\alpha)^{<\omega}$ not containing $h$ ($h \notin \ran(w)$),
    \item and any (non-empty) open interval $O \subseteq R \cap (r, \infty)$, so that $\langle w \restriction i\rangle[O] \subseteq (r, \infty)$, for each $i \leq \vert w \vert$,
\end{enumerate}
    there is $x \in X_\alpha \cap O$, with $\langle w \rangle(x) \in C \cup Y$. \footnote{Note that this implies that there is a dense (in $O$) set of such $x$, as we may shrink $O$ arbitrarily, retaining the property in (c). Also, note that $\langle w \rangle(x) \in Y$ happens exactly when $w$ can be written as $w_0 k w_1$ for some $(k,D,q) \in \mathcal{H}_\alpha$, where $\langle w_0 \rangle(x) \in D$. Otherwise, $\langle w\rangle(x) \in X_\alpha$, by the closure properties of (1) and (2), and $Y$ being closed under $\mathcal{G}$. In particular, for $x \in X_\alpha$, $\langle w \rangle(x) \notin Y$ if and only if $\langle w \rangle(x) \in X_\alpha$.}
    \end{enumerate}

    Before we start the construction, let us briefly explain the purpose of these objects. The Hagendorf order will be $X = \bigcup_{\alpha < \omega_1} X_\alpha$. $Y$ and the $Z_\alpha$'s consist of forbidden elements that should not be added to $X$. The functions $g \in \mathcal{G}$ will be self-embeddings of $X$. These are used to argue that certain suborders of $X$ are ``large", i.e. equimorphic with $X$. On the other hand, $(h, C,r) \in \mathcal{H}_\alpha$ witnesses that $X \setminus C$ embeds into an initial segment of $X$ (namely into $X \cap (0, r+1)$). We need to avoid that all of $X$ gets sent into that initial segment of $X$, hence $h$ maps the points in $C$ to forbidden values. This explains properties (1) and (2). Property (3) is much harder to motivate. It is another crucial element when we want to ensure that no function embeds $X$ into an initial segment and this is the most difficult part of the construction.

    For the construction, list all order preserving functions $f$ from the fact as $\langle f_{2\alpha+1}: \alpha < \omega_1 \rangle$. Next, enumerate $R$ as $\langle x_\xi : \xi < \omega_1 \rangle$. Let $\langle A_\alpha : \alpha < \omega_1 \rangle$ be a $\diamondsuit$-sequence on $\omega_1$, and for each $\alpha$, define $C_{2\alpha} = \{ x_\xi : \xi \in A_\alpha \}$ and $r_{2\alpha} = \max \{r \in \mathbb{N} : C_{2\alpha} \subseteq (r, \infty) \}$ (if $C_{2\alpha} = \emptyset$, just let $r = 0$). So, $\langle C_{2\alpha}, r_{2\alpha} : \alpha < \omega_1 \rangle$ lists pairs $(C,r)$ where $r \in \mathbb{N}$, and $C$ is a countable subset of $R \cap (r, \infty)$.  

    Start with $X_0$ any countable dense subset of $R$ disjoint from $Y$. Let $\mathcal{G}$ be countable so that for each $g \in \mathcal{G}$, $g[X_0] \subseteq X_0$, $g[Y] \subseteq Y$, and for any rationals $a <b <c<d$, $(a, b), (c, d) \subseteq R$, there is $g \in \mathcal{G}$ with $g[(a,b)] \subseteq (c,d)$. Since $X_0$ and $Y$ are dense, this can be achieved using an easy back-and-forth construction. We let $Z_0 = \mathcal{H}_0 = \emptyset$, so (2) and (3) above are satisfied vacuously. At limit steps, we simply take unions of the previous sets. The above properties are then clearly preserved. Thus, suppose that we are now given $X_\alpha$, $Z_\alpha$ and $\mathcal{H}_\alpha$. There are two cases how to proceed: 

    \underline{Case 1}: $\alpha = 1 \mod 2$. We let $\mathcal{H}_{\alpha +1} = \mathcal{H}_\alpha$, and it remains to say how we extend $X_\alpha$ and $Z_\alpha$. Property (3) above cannot be destroyed in this case. So let $f = f_\alpha$. 
    
    If the range of $f$ is cofinal in $\mathbb{R}$, or $X_\alpha$ is not contained in the domain of $f$, we may skip this case immediately, as $f$ will be of no relevance. Simply let $X_{\alpha+1} = X_\alpha$, $Z_{\alpha +1} = Z_\alpha$. If there is any $x \in X_\alpha$, with $f(x) \notin X_\alpha$, we let $Z_{\alpha +1} = Z_\alpha \cup \{ f(x) \}$, $X_{\alpha +1} = X_\alpha$.
    
    Otherwise, the range of $f$ is contained within $(-\infty, n)$, for some natural $n$. Our goal is to add some new value $x \in R$ to $X_\alpha$ which is either not in the domain of $f$, or we add $f(x)$ to $Z_\alpha$. This will ensure that $f$ will not be a self-embedding of our final order $X$. For a particular $x \in R$, the following are the only potential obstructions to this: 
        \begin{itemize}
            \item $x \in \dom f$ and $f(x) \in X_\alpha$, 
            \item there is a word  $w \in (\mathcal{G} \cup \dom \mathcal{H}_\alpha)^{<\omega}$, so that $\langle w \rangle(x) \in Y \cup Z_\alpha$ (e.g. if already $x \in Y \cup Z_\alpha$),
            \item $x \in \dom f$ and there is a word  $w \in (\mathcal{G} \cup \dom \mathcal{H}_\alpha)^{<\omega}$, so that $f(x) = \langle w \rangle(x)$.
        \end{itemize}

    If none of the above is satisfied for a particular $x \in R$, then if $X_{\alpha +1 } \supseteq X_\alpha \cup \{ x \}$ is smallest such that the closure properties (1) and (2) hold, we have that $X_{\alpha +1} \cap (Y \cup Z_\alpha) = \emptyset$ and, either $x$ is not in the domain of $f$, or $f(x) \notin X_{\alpha +1}$. Then let $Z_{\alpha +1} = Z_\alpha \cup \{ f(x)\}$.

    We will argue that indeed $x \in R$ can be found, for which none of the obstructions listed above occurs. Suppose not. There are clearly at most countably many $x$ for which one of the first two options can happen. For each word $w \in (\mathcal{G} \cup \dom \mathcal{H}_\alpha)^{<\omega}$, let $R_w = \{ x \in R : f(x) = \langle w \rangle(x) \}.$ By assumption, these sets cover a co-countable subset of $R$. Thus there must be some $w$ and some open interval $O \subseteq R \cap (n, \infty)$ so that $R_w$ is dense in $O$, since otherwise, $R\cap (n, \infty)$ is covered by countably many nowhere dense sets. This is impossible by the Baire Category Theorem. So the order-preserving function $f$ agrees with the continuous function $\langle w \rangle$ on a dense subset of $O$. It is easy to see that $f$ then already agrees with $\langle w \rangle$ on all of $O \cap \dom f$. As $O \subseteq (n, \infty)$ and $f$ maps into $(0, n)$, we can find a minimal initial subword $w' = w \restriction i$, $i \leq \vert w \vert$, of $w$, so that $\langle w' \rangle[O] \cap (0, n) \neq \emptyset$. This word $w'$ must be non-empty, as $O \cap (0, n) = \emptyset$, and its last letter must be a function $h \in  \dom \mathcal{H}_\alpha$, as all $g \in \mathcal{G}$ are progressive. Let $C, r$ be so that $(h, C, r) \in \mathcal{H}_\alpha$. Then $r < n$ must be the case as well, as otherwise $h$ is the indentity function on $(0, n)$ and maps no value above $n$ into $(0,n)$. Let $w' = w'' h$. By the minimality of $w'$, $h$ cannot appear in $w''$, as the whole range of $h$ is already contained in $(r, r+1) \subseteq (0, n)$. By property (3) above and by the minimality of $w'$, there is some $x \in X_\alpha \cap O$, with $\langle w'' \rangle (x) \in C \cup Y$ (every initial subword of $w''$ maps $O$ into $(n, \infty) \subseteq (r, \infty)$). But then $$\langle w' \rangle (x) = h(\langle w'' \rangle (x))  \in Y$$ and further, we have that $$\langle w \rangle(x) \in Y.$$ 
    As $x \in X_\alpha \subseteq \dom f$, $\langle w \rangle(x) = f(x) \in Y$, so $f(x) \notin X_\alpha$. We have already assumed that this is not the case -- contradiction.
 
    \underline{Case 2}: $\alpha = 0 \mod 2$. We let $X_{\alpha +1} = X_\alpha$, $Z_{\alpha +1} = Z_\alpha$ and it remains to say how we extend $\mathcal{H}_{\alpha +1}$. Property (1) above cannot be destroyed in this case. Let $(C, r) := (C_\alpha, r_\alpha)$. Then $C$ is a countable subset of $(r, \infty)$. If $C$ is not a subset of $X_\alpha$, we may skip this step immediately and let $\mathcal{H}_{\alpha +1} = \mathcal{H}_\alpha$. Otherwise, we have two subcases: 
\begin{itemize}
    \item \underline{Case 2a}: For any word $w \in (\mathcal{G} \cup \dom \mathcal{H}_\alpha)^{<\omega}$, and any open interval $O \subseteq R \cap (r, \infty)$, so that $\langle w \restriction i\rangle[O] \subseteq (r, \infty)$, for each $i \leq \vert w \vert$, there is $x \in X_\alpha \cap O$ with $\langle w \rangle(x) \in C \cup Y$.
    
  Note that $C$ is necessarily dense in $(r, \infty)$, by choice of $\mathcal{G}$.  We would like to find $h \colon R \to R \cap (0, r+1)$, for which we can extend $\mathcal{H}_\alpha$ to $\mathcal{H}_{\alpha +1} = \mathcal{H}_\alpha \cup \{(h,C,r)\}$, retaining the properties (2) and (3) above. As $X_\alpha$ and $Y$ are dense, a simple back-and-forth like construction allows us to find $h$, so that $h$ is the identity on $(0, r)$, $h[X_\alpha \setminus C] \subseteq X_\alpha$ and $h[Y \cup C] \subseteq Y$. In particular then, property (2) is preserved. We need to be more careful to ensure (3) is satisfied. So let us flesh out the construction in more detail: 

    Let $\langle x_n : n \in \omega \rangle$, $\langle c_n : n \in \omega \rangle$ and $\langle y_n : n \in \omega \rangle$ enumerate $(X_\alpha \setminus C) \cap (r, \infty)$, $C \cup (Y \cap (r, \infty))$ and $Y \cap (r, r+1)$ respectively. Further, let $$\langle (k_n, D_n, s_n, \bar w^n, m_n, O_n) : n \in \omega \rangle$$ enumerate all $(k, D, s, \bar w, m, O)$, where $(k, D, s) \in \mathcal{H}_\alpha$, $\bar w = (w_0, \dots, w_{m})$ is a sequence of words in $(\mathcal{G} \cup \dom \mathcal{H}_\alpha)^{<\omega}$ not containing $k$ and $O \subseteq R \cap (s, \infty)$ is an open interval with rational endpoints. We will step by step decide values of $h$ on $(r, \infty)$, preserving the order and mapping into $(r, r+1)$. At each step only finitely many values will have been decided. At step $n$, if $h(x_n)$ or $h(c_n)$ have not been decided yet, we may easily do so with $h(x_n) \in X_\alpha$, or $h(c_n) \in Y_\alpha$ respectively. Also if $y_n$ is not in the range of $h$ yet, we may stipulate that $h(c_k) = y_n$ for some suitable $k$, as $C$ is dense in $(r, \infty)$. This will ensure that the range of $h$ is dense in $(r, r+1)$, as $Y_\alpha$ is, so $h$ uniquely extends to a continuous function with domain $R$. Last we need to deal with $(k, D, s, (w_0, \dots, w_{m}), m, O) = (k_n, D_n, s_n, \bar w^n, m_n, O_n)$. 

   By the continuity of $\langle w_0 \rangle$, $\langle w_0 \rangle[O]$ is a bounded open interval $U_0 = (a_0, b_0)$. Decide $h(a_0)$ and $h(b_0)$ if these values haven't been decided yet. Then $V_1 = (h(a_0), h(b_0))$ is another open interval and so is $\langle w_1 \rangle[V_1] = U_1 = (a_1, b_1)$. Again, ensure that $h(a_1)$ and $h(b_1)$ are decided and let $V_2 = (h(a_1), h(b_1))$. Continue in this fashion until $U_m = (a_m, b_m)$ has been defined. If for some $j \leq \vert w_m \vert$, $\langle w_m \restriction j \rangle[V_m]$ is not contained in $(s, \infty)$, then also $$\langle w_0 h w_1 \dots{w_{m-1}} h (w_m \restriction j) \rangle[O]$$ will not be contained in $(s, \infty)$ and we are done for this step. Otherwise, by setting a few more values of $h$, we may find $x \in X_\alpha \cap O$ and $y$ so that $$\langle w_0 h w_1  \dots {w_{m-2}}  h {w_{m-1}} \rangle(x) = y$$ is decided, while $h(y)$ has not been decided yet. If either $y \in Y$ or $y \in C$, we are done as this will imply that $$\langle w_0 h w_1 \dots{w_{m-1}} h w_m\rangle(x) \in Y.$$ Otherwise, we can only have that $y \in X_\alpha \setminus C$ (see the previous footnote). Let $c$ be greatest below $y$ and $d$ be least above $y$, so that $h(c)$ and $h(d)$ are decided. We know that $\langle w_m \restriction j \rangle[(h(c), h(d))]$ is contained within $(s, \infty)$ for each $j \leq \vert w_m \vert$, as $(c, d) \subseteq U_{m-1}$, so $(h(c), h(d)) \subseteq V_m$. Thus, since property (3) above holds at stage $\alpha$, there is $x' \in X_\alpha \cap (h(c), h(d))$ so that $\langle w_m \rangle(x') \in D \cup Y$. Simply set $h(y) = x'$, thus $$\langle w_0 h w_1 \dots{w_{i-1}} h w_m\rangle(x) \in D \cup Y.$$

     \item \underline{Case 2b}: 2a does not hold. Then we simply let $\mathcal{H}_{\alpha + 1} = \mathcal{H}_\alpha$.
\end{itemize}

\medskip

This finishes the construction of $\langle X_\alpha, Z_\alpha, \mathcal{H}_\alpha : \alpha < \omega_1 \rangle$. As previously anticipated, we now claim that $X = \bigcup_{\alpha < \omega_1} X_\alpha$ is a Hagendorf ordering. Let $Z = \bigcup_{\alpha < \omega_1} Z_\alpha$ and $\mathcal{H} = \bigcup_{\alpha < \omega_1} \mathcal{H}_\alpha$.

First of all, $X$ does not embed into any of its initial segments. Otherwise, there is $\alpha < \omega_1$, with $f_{\alpha}[X] \subseteq X \cap (0, r)$ for some $r \in \mathbb{N}$. So $X_{\alpha} \subseteq X \subseteq \dom f_{\alpha}$ and the range of $f_{\alpha}$ is not cofinal in $\mathbb{R}$, as $X$ is dense in $\mathbb{R}$. But we have ensured that for some $x\in X$, $f(x) \in Z$, so $f(x) \notin X$.

Next, $X$ does embed into every final segment. Namely, if $r \in \mathbb{N}$, $r \geq 1$, find $g_n \in \mathcal{G}$ mapping $(n, n+1)$ into $(n+r, n+r+1)$,  for each $n \in \omega$. Let $g = \bigcup_{n \in \omega} g_n \restriction (n,n+1)$. Then $g$ is an order-preserving map with domain $R$ and $g[X] \subseteq X \cap (r, \infty)$. 

Finally, let $A \subseteq X$ be arbitrary. We would like to show that either, $A$ embeds into an initial segment of $X$, or, $X$ embeds into $A$. Put $B = X \setminus A$. There are two cases mirroring the previous case distinction: 

\begin{itemize}
    \item \underline{Case a}: There is $r \in \mathbb{N}$, so that for any word $w \in (\mathcal{G} \cup \mathcal{H})^{<\omega}$ and any open interval $O \subseteq R \cap (r, \infty)$, where $\langle w \restriction i \rangle[O] \subseteq (r, \infty)$, for each $i \leq \vert w \vert$, there is $x \in X \cap O$ with $\langle w \rangle(x) \in B \cup Y$.

    Then the set of limits $\alpha < \omega_1$, where the above holds, replacing $X$ with $X_{\alpha}$, $\mathcal{H}$ with $\mathcal{H}_{\alpha}$, and $B$ with $B_{\alpha} = \{ x_\xi \in B \cap (r, \infty) : \xi < \alpha  \}$, is clearly a club. Since $\langle A_\alpha : \alpha < \omega_1 \rangle$ is a $\diamondsuit$-sequence, there is some element $\alpha$ of that club, where $A_\alpha = \{ \xi < \alpha : x_\xi \in B \cap (r, \infty) \}$. So $B_{\alpha} = C_{\alpha}$, $r \leq r_\alpha$ and in stage $\alpha$ of the construction, we ended up in Case 2a. But then we have added $(h, B_\alpha, r_\alpha) \in \mathcal{H}$, where $h$ is order-preserving and $h[X \setminus B_\alpha] \subseteq X \cap (0, r_\alpha + 1)$. As $A = X \setminus B \subseteq X \setminus B_\alpha$, $A$ embeds into an initial segment of $X$.
    
    \item \underline{Case b}: Case a does not hold. This means that for every $r \in \mathbb{N}$, there is a word $w_r \in (\mathcal{G} \cup \mathcal{H})^{<\omega}$ and an open interval $O_r \subseteq R \cap (r, \infty)$, with $\langle w_r \rangle[O_r] \subseteq (r, \infty)$ and $\langle w_r \rangle[ O_r \cap X] \cap (B \cup Y) = \emptyset$. Recall that whenever $x \in X$ and $\langle w_r \rangle(x) \notin Y$, then $\langle w_r \rangle(x) \in X$. Hence, $\langle w_r \rangle[ O_r \cap X] \subseteq X \setminus B = A$.

    We may find an increasing sequence $\langle r_n : n \in \omega \rangle$ so that for each $n$, $r_n \geq n +1$ and $r_{n+1}$ is large enough such that $O_{r_n}, \langle w_{r_n} \rangle[O_{r_n}] \subseteq (r_n, r_{n+1})$. For each $n \in \omega$, let $g_n \in \mathcal{G}$ map $(n, n+1)$ into $O_{r_n}$. Finally, define $$g = \bigcup_{n \in \omega} \langle w_{r_n} \rangle \circ (g_n \restriction (n, n+1)).$$

    Then $g \colon R \to R$ is an order-preserving map and $g[X] \subseteq A$, so $X$ embeds into $A$.
\end{itemize}

This finishes the proof.\end{proof}






\section{A real Hagendorf order under \texorpdfstring{$\PFA$}{PFA}}\label{sec:pfa}

We show that a variation on the proof of Theorem~\ref{thm:diamond} works under $\PFA$. First, we need to consider a modification of Baumgartner's forcing for adding an isomorphism between $\aleph_1$-dense sets of reals. There is no essential difference except that there are two pairs of sets of reals $(A_0,B_0)$, $(A_1, B_1)$ that an order-preserving map is added for simultaneously. This will replace the back-and-forth construction in the previous proof. 

\begin{definition}\label{def:Baum}
    For $i\in 2$, let $A_i, B_i$ be sets of reals and let $\bar A_i = \langle A_{i, \beta} : \beta < \omega_1 \rangle$, $\bar B_i = \langle B_{i, \beta} : \beta < \omega_1 \rangle$ with $A_i = \bigcup_{\beta < \omega_1} A_{i, \beta}$, $B_i = \bigcup_{\beta < \omega_1} B_{i, \beta}$. The we write $\mathbb{P}(\bar A_0, \bar B_0, \bar A_1, \bar B_1)$ for the forcing notion consisting of finite partial order-preserving maps $p$ from $A_0 \cup A_1$ to $B_0 \cup B_1$, ordered by extension and where $p[A_{i, \beta}] \subseteq B_{i, \beta}$, for each $i \in 2$, $\beta < \omega_1$.
\end{definition}

We say that a set $Q$ is $\aleph_1$ dense in a non-empty open set $O$, if $\vert Q \cap U \vert = \aleph_1$ for every non-empty open $U \subseteq O$.

\begin{lemma}[essentially Baumgartner \cite{Baumgartner}]\label{lem:Baum}
    Assume $\CH$. Let $A_0, A_1 \subseteq \mathbb{R}$ be disjoint and $B_0, B_1$ be disjoint $\aleph_1$-dense subsets of some open subset of $\mathbb{R}$. Further, let $Q_\beta \subseteq B_0$ be $\aleph_1$-dense in an open set $O_\beta$, for each $\beta < \omega_1$. Then there are sequences $\bar A_i$, $\bar B_i$ as in Definition~\ref{def:Baum} such that \begin{itemize}
        \item for each $\beta < \omega_1$, $B_{0, \beta}$, $B_{1,\beta}$ are countable dense,
        \item for each $\beta < \omega_1$ and all $\beta' \in [\beta, \omega_1)$, $Q_{\beta} \cap B_{0, \beta'}$ is dense in $O_{\beta}$, 
        \item $\mathbb{P}(\bar A_0, \bar B_0, \bar A_1, \bar B_1)$ is ccc. 
        
    \end{itemize}
\end{lemma}

For a proof, we refer directly to Baumgartner's original argument or the excellent modern exposition given in \cite{Vartanian}. There is nothing of difficulty in generalizing Baumgartner's construction to the above.


\begin{thm}\label{thm:PFA=>Hagendorf_real}
    Assuming $\PFA$, there is a Hagendorf order $X \subseteq \mathbb{R}$.
\end{thm}

Note that under $\PFA$, a Hagendorf suborder of the reals must have size $\aleph_2$ by Theorem~\ref{thm:BA=>noaleph1}.

\begin{proof}
    We proceed almost exactly as in the proof of Theorem~\ref{thm:diamond=>Hagendorf_real}. There a few important differences though. First, as one can probably guess, we replace $\omega_1$ with $\omega_2$, ``countable" with ``of size $\aleph_1$" and ``dense" with $\aleph_1$-dense everywhere. Next, instead of fixing a single collection $\mathcal{G}$ throughout the construction, we will have to progressively add more $g$'s. Thus this time, we will construct $\langle X_\alpha, Z_\alpha, \mathcal{H}_\alpha, \mathcal{G}_\alpha : \alpha < \omega_2 \rangle$. In the proof of Theorem~\ref{thm:diamond=>Hagendorf_real}, at each step $\alpha$, simply replace each occurrence of ``$\mathcal{G}$" with ``$\mathcal{G}_\alpha$". The most important difference occurs with Property (3), which will now read as: 

    \begin{enumerate}
    \setcounter{enumi}{2}
        \item For any \begin{enumerate}
    \item $(h, C, r) \in \mathcal{H}_\alpha$,
    \item word $w \in (\mathcal{G}_\alpha \cup \dom \mathcal{H}_\alpha)^{<\omega}$ not containing $h$,
    \item and any open interval $O \subseteq R \cap (r, \infty)$, so that $\langle w \restriction i\rangle[O] \subseteq (r, \infty)$, for each $i \leq \vert w \vert$,
    \end{enumerate}
    the set of $x \in X_\alpha \cap O$ with $\langle w \rangle(x) \in C \cup Y$ is \emph{$\aleph_1$-dense} in $O$.
    \end{enumerate}

    There are no other notable changes in the rest of the setup. $\PFA$ implies $\diamondsuit_{\omega_2}$ (see e.g. \cite[Theorem 5.8]{Devlin}), thus a suitable diamond sequence exists. Just as with $Z_0$ and $\mathcal{H}_0$ we may pick $\mathcal{G}_0$ to be empty. 
    
    Next, instead of dividing the cases into odd and even $\alpha$, we will divide into three cases according to $\alpha = 0,1,2 \mod 3$, as we need extra steps to extend the $\mathcal{G}_\alpha$'s. To this end, there will be a sequence $\langle n_\alpha, O_\alpha, E_{\alpha}:\alpha = 2 \mod 3, \alpha < \omega_2 \rangle$ enumerating all triples $(n, O, E)$, where $n \in \omega$, $O \subseteq R$ is an open interval with $n+1 \leq \inf O$ and $E \subseteq O$ is countable. 

    Case 1 now corresponds to $\alpha = 1 \mod 3$ and works in the same manner as before. The only obstacle we encounter is that we cannot apply the Baire Category Theorem anymore. On the other hand, $\PFA$ implies that the union of $\aleph_1$-many nowhere dense sets is meager. Thus there is no issue here and we may proceed. 

    Case 2 corresponds to $\alpha = 0 \mod 3$ (in particular, there is a club of such stages). The diamond sequence hands us a pair $(C,r)$ and we may assume that $C \subseteq X_\alpha$. As before, we split into two subcases, this time, mirroring our modified property (3): 

    \medskip

\underline{Case 2a}: For any word $w \in (\mathcal{G}_\alpha \cup \dom \mathcal{H}_\alpha)^{<\omega}$, and any open interval $O \subseteq R \cap (r, \infty)$, so that $\langle w \restriction i\rangle[O] \subseteq (r, \infty)$, for each $i \leq \vert w \vert$, there is an $\aleph_1$-dense (in $O$) set of $x \in X_\alpha \cap O$ with $\langle w \rangle(x) \in C \cup Y$.

        We would like find a function $h$ just as before and add $(h, C, r)$ to $\mathcal{H}_{\alpha}$. This time we cannot rely on a back and forth argument. Instead we use Lemma~\ref{lem:Baum} and apply $\PFA$. More specifically, let $A_0 = (X_\alpha \setminus C) \cap (r, \infty)$, $B_0 = X_\alpha \cap (r, r+1)$, and $A_1 = C \cup Y  \cap (r, \infty)$, $B_1 = Y \cap (r, r+1)$.  Note that $B_0$ and $B_1$ are $\aleph_1$-dense in $(r, r+1)$. Further, we let $\langle Q_\beta, O_\beta : \beta < \omega_1 \rangle$ enumerate all pairs $(Q, O)$, where $O \subseteq (r, r+1)$ is an open interval with rational endpoints, $Q \subseteq B_0$ is $\aleph_1$-dense in $O$, and for some $w \in (\mathcal{G}_\alpha \cup \dom \mathcal{H}_\alpha)^{<\omega}$ and some $D$ appearing as the second coordinate of some element of $\mathcal{H}_\alpha$, $Q$ is the set of $x \in O \cap X_\alpha$, with $\langle w \rangle(x) \in D\cup Y$.

        Let $\mathbb{P}$ be a $\sigma$-closed poset forcing $\CH$. In any $\mathbb{P}$-generic forcing extension, we find $\bar A_0, \bar A_1, \bar B_0, \bar B_1$ as in Lemma~\ref{lem:Baum}. So let $\dot{\mathbb{Q}}$ be a $\mathbb{P}$-name for a suitable forcing $\mathbb{P}(\bar A_0, \bar A_1, \bar B_0, \bar B_1)$. Then $\mathbb{P} * \dot{\mathbb{Q}}$ is proper. A standard genericity argument and an application of $\PFA$ yields an order-preserving function from $(r, \infty)$ to $(r, r+1)$ mapping $A_0$ into $B_0$ and $A_1$ into $B_1$.\footnote{Note that $A_1$ is dense in $(r, \infty)$ whereby the generic order-embedding can be uniquely extended to a total and continuous function on $(r, \infty)$.} We need to ensure our modified property (3) is preseved when adding $(h, C, r)$ to $\mathcal{H}_\alpha$.\footnote{Recall that $h \colon R \to R \cap (0, r+1)$ is set to be the identity on $R \cap (0, r)$ and corresponds to the function we are trying to construct on $R \cap (r, \infty)$.} This is another genericity argument and it is here that the modification of (3) and the use of $\langle Q_\beta, O_\beta : \beta < \omega_1 \rangle$ become relevant. 
        
        Let us work within a $\mathbb{P}$-generic extension for now and let $\mathbb{Q} = \mathbb{P}(\bar A_0, \bar A_1, \bar B_0, \bar B_1)$. The argument we will make can easily be translated into a genericity argument for $\mathbb{P} * \dot{\mathbb{Q}}$. Let $(k, D, s, \bar w, m, O)$ be as in the proof of Theorem~\ref{thm:diamond=>Hagendorf_real}. Specifically, $(k, D, s) \in \mathcal{H}_\alpha$, $\bar w = (w_0, \dots, w_m)$ is a sequence of words in $(\mathcal{G}_\alpha \cup \mathcal{H}_\alpha)^{<\omega}$ not containing $k$ and $O \subseteq R \cap (s, \infty)$ is an open interval, say with endpoints in the dense set $A_1$. Let $p \in \mathbb{P}(\bar A_0, \bar A_1, \bar B_0, \bar B_1)$ be arbitrary. For a given countable subset $I$ of $X_\alpha \cap O$, we would like to find $q \leq p$ and $x \in (X_\alpha \cap O) \setminus I$ so that $q$ forces that $$\langle w_0 h w_1  \dots {w_{m-1}}  h {w_{m}} \rangle(x) \in D \cup Y$$ if \begin{equation}\tag{$*$}
            \langle (w_0 h w_1  \dots {w_{m-1}}  h {w_{m}}) \restriction j \rangle[O] \subseteq (s, \infty), \text{ for each $j$.}\footnote{Note that it is not necessary to make a separate genericity argument for each of the $\aleph_2$-many countable subsets $I$ of $X_\alpha \cap O$, since $\aleph_1$-many such sets suffice to cover them.}
        \end{equation}  Towards this end, we may assume that $p$ already forces $(*)$. Similar to the situation in Theorem~\ref{thm:diamond=>Hagendorf_real}, we find $p_0 \leq p$ so that $p_0$ forces that $$\langle w_0 h w_1  \dots {w_{m-2}}  h {w_{m-1} h} \rangle[O] = V,$$ for some open set $V \subseteq (r, r+1)$. By property (3), there is an $\aleph_1$-dense set of $z \in X_\alpha \cap V$ with $\langle w_m \rangle(z) \in D \cup Y$. This set is of the form $Q_\beta$, for some $\beta < \omega_1$, where $O_\beta = V$. Let $\beta^*$ be large enough, such that for all $\beta' \geq \beta^*$, $Q_{\beta} \cap B_{0, \beta'}$ is dense in $O_{\beta}$. By the ccc, we can find a countable set $J \subseteq X_\alpha \cap O$ such that it is forced that $$\{ x \in X_\alpha \cap O : \langle w_0 h w_1 \dots {w_{m-2}} h {w_{m-1}} \rangle(x) \in \bigcup_{\beta' < \beta^*} A_{0, \beta'} \} \subseteq J.$$

        Now we can find $x \in (X_\alpha \cap O) \setminus (I \cup J)$ and $p_1 \leq p_0$ so that $p_1$ decides that $\langle w_0 h w_1  \dots {w_{m-2}}  h {w_{m-1}} \rangle(x) = y$, where $y\notin \dom p_1$. Here, note that this value is forced to be a member of $X_\alpha \cup Y$. If $ y \in C \cup Y$ (e.g. when $x \in C$), we can let $q = p_1$ and we are done. Otherwise, $y \in A_{0, \beta'}$ for some $\beta' > \beta^*$, as $x \notin J$. Since $Q_{\beta} \cap B_{0, \beta'}$ is dense in $V$, we can easily extend to $q \leq p_1$, where $q(y) \in Q_{ \beta}$, and we are done once again.

        \underline{Case 2b}: 2a does not hold. Just as before, we can ignore this case.

        \medskip

        Finally, we have:

        \medskip
        
        \underline{Case 3:} $\alpha = 2 \mod 3$. We are only extending $\mathcal{G}_\alpha$ by a single function $g$ in this case. We are given the triple $(n, O, E) = (n_\alpha, O_\alpha, E_\alpha)$, and our goal is to produce a continuous progressive function $g\colon R \to R$, with $g[Y] \subseteq Y$, $g[X_\alpha] \subseteq X_\alpha \setminus E$ and $g[(n, n+1)] \subseteq O$. In doing so we must preserve property (3). This can be achieved almost identically to Case 2a above. Note specifically, that, if $Q \subseteq X_\alpha$ is $\aleph_1$-dense in an open set $O$, so is $Q \setminus E$, as $E$ is countable.  

        \medskip
        
        This finishes the construction of the order $X = \bigcup_{\alpha < \omega_2} X_\alpha$. Just as before, $X$ does not embed into any proper initial segment but does embed into every final segment, thanks to the functions in $\mathcal{G} = \bigcup_{\alpha < \omega_2} \mathcal{G}_\alpha$. For arbitrary $A \subseteq X$, again, let $B = X \setminus A$. 

        \medskip

\underline{Case a}: There is $r \in \mathbb{N}$, so that for any word $w \in (\mathcal{G} \cup \mathcal{H})^{<\omega}$ and any open interval $O \subseteq R \cap (r, \infty)$, where $\langle w \restriction i \rangle[O] \subseteq (r, \infty)$, for each $i \leq \vert w \vert$, there are $\aleph_1$-densely-many $x \in X \cap O$ with $\langle w \rangle(x) \in B \cup Y$.

      \noindent An analogous guessing argument shows that $A$ embeds into $X \cap (0, r+1)$.
      
\underline{Case b}: For each $r \in \mathbb{N}$, there is a word $w_r \in (\mathcal{G} \cup \mathcal{H})^{<\omega}$ and an open interval $O_r \subseteq R \cap (r, \infty)$, with $\langle w_r \rangle[O_r] \subseteq (r, \infty)$ and $\langle w_r \rangle[ O_r \cap X] \cap (B \cup Y)$ is countable.

      \noindent Let $E_r$ denote the set of $x \in O_r \cap X$ with $\langle w_r \rangle(x) \in B \cup Y$, which is countable itself. This time, for every $n$ and $r \geq n+1$, we find a function $g \in \mathcal{G}$ mapping $(n, n+1)$ into $O_r \setminus E_r$. Again, proceed as before. 
\end{proof}

\section{Open questions}

\begin{quest}
    Is it consistent that there are no real Hagendorf types? Is it consistent with $\CH$?
\end{quest}

We remark that, due to Fact~\ref{fact:borel}, the statement that $X \subseteq \mathbb{R}$ is not Hagendorf can be expressed by a formula of the form $\exists A \subseteq X \varphi(A,X)$, where $\varphi$ is projective. It follows by a general absoluteness argument that, if $X \subseteq \mathbb{R}$ is Hagendorf of size $\mathfrak{c}$, then this cannot be destroyed with $<\mathfrak{c}$-closed forcing.

Another option is to add a subset $A \subseteq X$ with finite approximations. It is not difficult to see that in the generic extension we have $A < X$.  Namely, any order-embedding of $X$ into $A$ must be covered by a Borel order-embedding $f$ appearing in $V[A \cap C]$, for some countable $C \in V$. But then generically $f(x) \notin A$, for some $x \in X \setminus C$. Unfortunately, a similar argument for showing that $A$ does not embed in an initial segment fails precisely because of the obstruction Theorem~\ref{thm:smallembed} poses.

This brings us to the following line of thought which might lead somewhere:

\begin{lemma}
   Let $X = \{ x_\xi : \xi < \omega_1 \} \subseteq \mathbb{R}$ be Hagendorf and suppose that there is a normal Suslin tree $T \subseteq 2^{<\omega_1}$ such that $X \setminus \{ x_\xi : s(\xi) = 0 \} \not< X$ for each $s \in T$.\footnote{The tree is \emph{normal} if for every $s \in T$, $s^\frown 0, s^\frown 1 \in T$ and for $\operatorname{lth}(s) \leq \xi < \omega_1$, there is $t \supseteq s$, $t \in T$, $\operatorname{lth}(t) = \xi$.} Then forcing with $T$ destroys $X$.
\end{lemma}

\begin{proof}
    Let $b$ be a generic branch through $T$ and let $A := \{ x_{\xi} : b(\xi) = 1 \}$. Since $T$ is $\sigma$-distributive, no new reals are added and we conclude that $A < X$, just as in the argument above. Suppose that $f$ embeds $A$ in a proper initial segment $(-\infty, x)$ of $X$ and this is forced by $s \in T$. Now note that $X \setminus \{ x_\xi : s(\xi) = 0 \}$ must also embed in that same segment, contradicting the assumption on the tree $T$.
\end{proof}

\begin{quest}
    Assume $\diamondsuit_{\omega_1}$ and $X\subseteq \mathbb{R}$ Hagendorf. Does a Suslin tree as above exist? 
\end{quest}

\begin{quest}
    Does the existence of a real Hagendorf order under $\CH$ imply the existence of a Suslin tree?
\end{quest}

The following question is related to the proof under $\PFA$:

\begin{quest}
    Is it consistent that there is a real Hagendorf order of size $\mathfrak{c} \geq \aleph_2$, where Theorem~\ref{thm:smallembed} is witnessed by a countable set?
\end{quest}

\begin{quest}[J. Larson]
    Do $\sigma$-scattered Hagendorf orders exist?
\end{quest}

\begin{quest}
    Is it consistent that there are no Hagendorf orders at all? 
\end{quest}

\subsection*{Acknowledgements}
This research was funded in whole or in part by the Austrian Science Fund (FWF) [10.55776/ESP5711024]. For open access purposes, the authors have applied a CC BY public copyright license to any author-accepted manuscript version arising from this submission. No data are associated with this article.

\bibliographystyle{plain}
\bibliography{ref}

\end{document}